%% file: bondal5d.tex
\documentclass{article}

\input{bondal-preamble}

\begin{document}

\title{Bondal's conjecture in dimension five}

\author{St\'ephane Druel\thanks{CNRS/Universit\'e Claude Bernard Lyon 1, \email{stephane.druel@math.cnrs.fr}} \and
Jorge Vit\'orio Pereira\thanks{IMPA, \email{jvp@impa.br}} \and
Brent Pym\thanks{McGill University, \email{brent.pym@mcgill.ca}} \and
Fr\'ed\'eric Touzet\thanks{Universit\'e de Rennes, \email{frederic.touzet@univ-rennes.fr}}}

\maketitle

\begin{abstract}
    Bondal's conjecture in Poisson geometry gives lower bounds on the degeneracy loci of Poisson Fano manifolds, where the rank of the Poisson structure drops.  By work of several authors, it was previously known to hold for Fano manifolds of dimension at most four.  We give the first proof of this conjecture for Fano manifolds of dimension five, and partial results for Fano manifolds of all odd dimensions.  The proof uses: (i) an algebraic integrability criterion for codimension-one foliations on weak Fano manifolds, extending a previous result of the first author; (ii) the ``modular residues'' of Poisson structures introduced by Gualtieri and the third author; and (iii) a cohomological constraint on invariant subvarieties for Pfaff fields, extending earlier results of Esteves--Kleiman to the case in which the Pfaff distribution on the subvariety admits a closed strongly directed positive current.
\end{abstract}

\tableofcontents

\section{Introduction}

Let $\X$ be a complex manifold and let $\sigma$ be a holomorphic Poisson structure on $\X$, i.e.~a global bivector field $\sigma\in \coH[0]{\X,\wedge^2\cT{\X}}$ such that $[\sigma,\sigma]=0$, where $[-,-]$ is the Schouten bracket.  Recall that $\sigma$ induces a singular foliation of $\X$ by symplectic leaves. We denote by $\Dgn[2k]{\sigma}$ the \defn{degeneracy locus} along which the rank of $\sigma$ is at most $2k$, i.e.~the union of the leaves of dimension at most $2k$.  We recall that $\Dgn[2k]{\sigma}$ is the vanishing locus of the Pfaffian
\[
\sigma^{k+1} \in \coH[0]{\X,\wedge^{2k+2}{\cT{\X}}},
\]
so that it naturally has the structure of a closed analytic subspace of $\X$.

A guiding expectation, due to Bondal, is that when the curvature of $\X$ is positive, the degeneracy loci are much larger than the classical formulae for dimensions of skew-symmetric determinantal varieties would predict. Namely, he conjectured the following.
\begin{conjecture}[{\cite{Bondal1993}, Bondal 1993}]
    Let $\X$ be a Fano manifold, and let $\sigma$ be a Poisson structure on $\X$.  Then $\Dgn[2k]{\sigma}$ has an irreducible component of dimension at least $2k+1$ for all integers $k$ such that $0 \leq 2k < \dim \X$.
\end{conjecture}

The following cases of the conjecture are known.  

\begin{enumerate}
    \item\label{it:polishchuk} For odd-dimensional manifolds, say $\dim \X=2n+1$, Polishchuk~\cite{Polishchuk1997} proved that Bondal's conjecture is true for the ``maximal'' degeneracy locus $\Dgn[2n-2]{\sigma}$.
    \item\label{it:beauville} Beauville~\cite{Beauville2011} observed that Polishchuk's argument applies more generally to prove the conjecture for $\Dgn[2r-2]{\sigma}$, where $2r$ is the maximal rank of $\sigma$, independent of the dimension of $\X$.
    \item\label{it:even-dim} When $\dim \X =2n$ is even, the ``maximal'' degeneracy locus $\Dgn[2n-2]{\sigma}$ is the vanishing locus of the anticanonical section $\sigma^n$, and hence the conjecture holds for $\Dgn[2n-2]{\sigma}$ simply because the anticanonical bundle of a Fano manifold is nontrivial.  In \cite{Gualtieri2013}, Gualtieri and the third author proved that the conjecture also holds for the ``submaximal'' degeneracy locus $\Dgn[2n-4]{\sigma}$.
\end{enumerate}
In particular, these results together imply that Bondal's conjecture is true for all Fano manifolds of dimension at most four.

Our main result in this paper establishes the conjecture for the submaximal degeneracy locus of odd-dimensional Poisson Fano manifolds, and thus gives the first proof of the conjecture when $\dim \X = 5$:

\begin{theorem}[see \autoref{thm:bondal-odd}]\label{theorem:main}
    Let $\X$ be a Fano manifold of dimension $2n+1$ with $n>1$, and let $\sigma$ be a Poisson structure on $\X$. Then $\Dgn[2n-4]{\sigma}$ has an irreducible component of dimension at least $2n-3$.
\end{theorem}

\begin{corollary}
    Bondal's conjecture holds for all Fano manifolds of dimension five.
\end{corollary}

Let us give a brief overview of the argument.  First, note that by case \ref{it:beauville} above, we may assume without loss of generality that the generic rank of $\sigma$ is equal to $2n$.  We then consider the holomorphic foliation $\cF$ defined by saturating the image of the contraction map induced by $\sigma$ (see \autoref{SS:Poissonstructures});  it agrees with the symplectic foliation defined by $\sigma$ where the latter has maximal rank, but in general, its leaves are unions of symplectic leaves of $\sigma$. 

There are two main situations to consider.  The first is that the singular locus of $\cF$ has codimension at least three.  In this situation, we prove that the closure of a generic leaf is a Fano manifold of dimension $2n$.  Thus, in this situation, \autoref{theorem:main} follows from case \ref{it:even-dim} above.  In fact, we prove a more general statement about codimension-one foliations with mild singularities on weak Fano manifolds, similar to a result of the first author~\cite{DruelWeakFano} for regular foliations of arbitrary codimension:
\begin{theorem}[see \autoref{theorem:fibration}]\label{thm:fibration-intro}
    Let $\X$ be a weak Fano manifold, and let $\cF$ be a codimension-one holomorphic foliation of $\X$ whose singular locus has codimension at least three.  Then $\cF$ is the foliation defined by the fibres of a morphism $\X \to \PP^1$ with irreducible fibres.
\end{theorem}

The second situation to consider is that the singular locus of $\cF$ has codimension two (the maximal possible).  In this situation, we treat several subcases depending on the dimensions of various subvarieties of the singular locus of $\cF$, using two main tools.  The first tool is the theory of ``modular residues'' of Poisson structures, introduced by Gualtieri and the third author~\cite{Gualtieri2013} to treat case \ref{it:even-dim} above; we recall the key aspects in \autoref{sec:modres}.  The second tool is an analogue of Esteves--Kleiman's work on Pfaff fields on schemes~\cite{EstevesKleiman} in the setting of compact K\"ahler manifolds; see \autoref{lemma:pfaff-kahler}, \autoref{prop:esteves_kleiman} and \autoref{prop:leaves-hodge}.  The latter gives cohomological constraints on the existence of Poisson subvarieties with sufficiently mild singularities.  A typical statement of this type is the following, which generalizes our result \cite[Theorem 1.2]{globalweinstein} on compact symplectic leaves to closures of symplectic leaves whose boundaries are not too big:
\begin{proposition}
    Let $\X$ be a compact K\"ahler manifold, and let $\sigma$ be a holomorphic Poisson structure on $\X$.  Suppose that $\X$ contains closed subvarieties $\Si\subset \Y\subset\X$ such that $\Y\setminus \Si$ is a symplectic leaf, and $\codim(\Si,\Y) \ge 2$.  Then $\coh[0,\dim \Y]{\X} > 0$.
\end{proposition}
More generally, we obtain a similar statement for smooth enough Poisson subvarieties whose foliations support a strongly directed positive current.  In the context of the proof of \autoref{theorem:main}, the latter are obtained using Weinstein's modular vector field.

\subsection{Acknowledgements}
The research described here was initiated at the workshop ``Foliations, birational geometry and applications'', held at the Centre international de rencontres math\'ematiques (CIRM) in February, 2025.  We thank CIRM and the workshop organizers for the stimulating event and environment.
J.V.P. and S.D.~acknowledge support from the CAPES-COFECUB
project Ma1017/24, funded by the French Ministry for Europe
and Foreign Affairs, the French Ministry for Higher Education and CAPES.  J.V.P.~was also supported by CNPq (PQ Scholarship 304690/2023-6 and Projeto
Universal 408687/2023-1 ``Geometria das Equações Diferenciais Algébricas''), and FAPERJ (Grant number E26/200.100/2026). B.P.~was supported by the Natural Sciences and Engineering Research Council of Canada (NSERC), through Discovery Grant RGPIN-2020-05191. F.T.~would like to thank the Henri Lebesgue Center for constant support.

\section{Foliations on weak Fano manifolds}\label{section:druel}
In this section we review some basic facts about foliations and their intersection numbers, and apply them to prove \autoref{thm:fibration-intro} from the introduction.

\subsection{Foliations}
Let $\X$ be a complex manifold.  By a \defn{foliation} $\cF$ of dimension $p$ on $\X$, we mean a saturated, involutive subsheaf $\cT{\cF} \subset \cT{\X}$ of rank $p$.  Its \defn{codimension} is $q = \dim \X - p$.  

The \defn{conormal sheaf} of $\cF$ is the annihilator $\coN{\cF}\subset\forms[1]{\X}$ of $\cT{\cF}$.  We will mainly be interested in the case $q=1$, in which case $\coN{\cF}$ is an invertible sheaf (a line bundle).  The \defn{normal sheaf} of $\cF$ is $\cN{\cF}:=(\coN{\cF})^\vee$; by definition, it is a torsion-free sheaf of rank $q$.

Let $\X^\circ \subseteq \X$ be the open set where $\cT{\cF}$ is a subbundle of $\cT{\X}$. The \defn{singular locus} of $\cF$ is $\sing(\cF):=\X \setminus \X^\circ$. We say that $\cF$ is \defn{regular} if its singular locus is empty.

\subsection{Intersection numbers}
Let $\cF$ be a codimension-one foliation on a complex manifold $\X$.  Our aim in this subsection is to give some constraints on the intersection number
\[
\cN{\cF}^{\dim \Y} \cdot \Y := \int_{\Y}c_1(\cN{\cF})^{\dim \Y}
\]
where $\Y\subset \X$ is a compact subvariety, with a particular focus on the case in which $\Y$ is a curve.

The intersection number will depend on whether the inclusion $\Y\hookrightarrow \X$ is tangent to $\cF$, in the following sense.

\begin{definition}
A morphism $\phi : \Y \to \X$ is \defn{tangent to $\cF$} if the composition of the natural maps
\[
\begin{tikzcd}
\phi^*\coN{\cF} \ar[r] & \phi^*\forms[1]{\X} \ar[r] & \forms[1]{\Y}
\end{tikzcd}
\]
is identically zero.  
\end{definition}
Recall that when $\cF$ is regular, $\cN{\cF}$ carries a flat partial connection in the direction of $\cT{\cF}$ (the Bott connection), which implies that the class $c_1(\cN{\cF})^q \in \coH[q]{\X,\forms[q]{\X}}$ lifts to an element of $\coH[q]{\X,\det \coN{\cF}}$.  As observed by the first author in \cite{DruelCodimensiontwo}, 
this phenomenon persists for singular foliations whose singular locus is sufficiently small, simply because the cohomology of locally free sheaves is invariant under restriction to complements of closed subvarieties of sufficiently high codimension.  Namely, we have the following.

\begin{lemma}[{\cite[Lemma 3.6]{DruelCodimensiontwo}}]
Let $\X$ be a complex manifold and let $\cF$ be a holomorphic foliation of codimension $q$ whose singular locus has codimension at least $q+2$. Then $c_1(\cN{\cF})^q \in \coH[q]{\X,\forms[q]{\X}}$ lies in the image of the natural map $\coH[q]{\X,\det\coN{\cF}} \to \coH[q]{\X,\forms[q]{\X}}$.
\end{lemma}

When applied to the case in which the codimension $q$ is equal to one, the lemma immediately gives the following numerical constraints on curves tangent to $\cF$:
\begin{corollary}\label{lemma:connection-normal-along-curve}
    Let $\cF$ be a codimension-one foliation on $\X$ whose singular locus has codimension at least three, and let 
    $\phi : \C \to \X$ be a compact curve tangent to $\cF$.  Then $\phi^*c_1(\cN{\cF}) = 0 \in \coH[1]{\C,\forms[1]{\C}}$.  In particular, we have the vanishing of the intersection number
    \[
    \cN{\cF} \cdot \C = 0
    \]
    of $\C$ and $\cN{\cF}$.
\end{corollary}

On the other hand, for rational curves that are not tangent to $\cF$, we have the following:

\begin{lemma}\label{lemma:noninvariant-rational-curve}
Let $\cF$ be a codimension-one holomorphic foliation, and let $\phi : \C \to \X$ be a rational curve that is not tangent to $\cF$. Then we have  
    \[
    \cN{\cF} \cdot \C \ge 2.
    \]
\end{lemma}

\begin{proof}
    By passing to the normalization of $\C$, we may assume without loss of generality that $\C$ is smooth.  By assumption, the composition $\phi^*\coN{\cF} \to \phi^*\forms[1]{\X}\to\forms[1]{\C}$ is nonzero.  Dually, we have a nonzero map $\cT{\C}\to \phi^*\cN{\cF}$, so the result follows from the definition $\cN{\cF}  \cdot \C:= \deg(\phi^*\cN{\cF})$ and the fact that $\deg(\cT{\C}) = 2$.
\end{proof}

\subsection{Foliations on weak Fano varieties}

Recall that a \defn{weak Fano manifold} is a complex projective manifold $\X$ whose anticanonical class $-K_\X$ is big and nef.  In \cite{DruelWeakFano}, the first author showed that regular foliations of arbitrary codimension on weak Fano manifolds are fibrations, i.e.~their generic leaves are the fibres of a morphism.  In this subsection, we show that the same conclusion holds for codimension-one foliations on weak Fano manifolds whose singular locus has sufficiently high codimension, namely we prove the following result (stated as \autoref{thm:fibration-intro} in the introduction).

\begin{theorem}\label{theorem:fibration}
    If $\X$ is a weak Fano projective manifold and $\cF$ is a codimension-one foliation on $\X$ with $\codim\,\sing(\cF)\ge 3$, then $\cF$ is defined by a morphism $f:\X\to \mathbb P^1$ with irreducible fibres.
\end{theorem}
We will construct the desired morphism by contracting suitable curves in $\X$.  To this end, we require a preparatory lemma.  Let $N_1(\X)$ be the $\RR$-vector space of one-cycles in $\X$ modulo numerical equivalence, and recall that the Mori cone $\NE(\X) \subset N_1(\X)$ is the set of non-negative linear combinations of classes of reduced, irreducible, proper curves.  The following fact about weak Fano manifolds is standard; we include a proof for lack of a reference.
\begin{lemma}\label{lem:mori-cone}
    If $\X$ is a weak Fano manifold, then $\NE(\X)$ is generated by classes of rational curves.
\end{lemma}
\begin{proof}
By Kodaira's Lemma, there exists an effective $\mathbb{Q}$-divisor $E$ such that $-K_\X-\varepsilon E$ is ample for any $\varepsilon \in \left ] 0,1 \right ]$. If $\varepsilon$ is a small enough positive rational number, then the pair $(\X,\varepsilon E)$ has klt singularities and $-(K_\X+\varepsilon E)$ is ample. In other words, $(\X,\varepsilon E)$ is log Fano. The result now follows from the cone theorem for klt pairs (see \cite[Theorem 3.7]{KollarMori}).
\end{proof}

\begin{proof}[Proof of \autoref{theorem:fibration}]
    By \autoref{lem:mori-cone}, the Mori cone $\NE(\X)$ is generated by classes of rational curves.     Let $\C\subset \X$ be a rational curve. If $\C$ is tangent to $\cF$, then  \autoref{lemma:connection-normal-along-curve} gives
    $\cN{\cF}\cdot \C=0$. If $\C$ is not tangent, then \autoref{lemma:noninvariant-rational-curve} gives
    $\cN{\cF}\cdot \C\ge 2$. Hence $\cN{\cF}\cdot \C\ge 0$ for every rational curve $\C$, and therefore for every curve class in $\NE(\X)$.
    Thus $\cN{\cF}$ is nef.

    Since $\cN{\cF}$ is nef and $-K_{\X}$ is nef and big, for any positive integer $a$ the divisor
    $a\cN{\cF}-K_{\X}=a\cN{\cF}+(-K_{\X})$ is nef and big. By the base point free theorem \cite[Theorem 3.3]{KollarMori}, some multiple of $\cN{\cF}$ is globally generated.
    Hence there exists a surjective morphism with connected fibres
    \[
        f:\X \to \Y
    \]
    onto a normal projective variety $\Y$, and an ample line bundle $\cA$ on $\Y$, such that
    \[
        \cN{\cF}^{\otimes m} \simeq f^*\cA
    \]
    for some integer $m \ge 1$.
    By the Baum--Bott residue theorem \cite{BaumBott}, the cohomology class $c_1(\cN{\cF})^2$ is represented by a cycle supported on the codimension-two components of $\sing(\cF)$. Since $\codim \sing(\cF)\ge 3$, there are no such components and therefore
    \[
        c_1(\cN{\cF})^2 = 0.
    \]
    On the other hand, if $\dim \Y\ge 2$ then, since $\cA$ is ample, one has $c_1(\cA)^2\neq 0$, and hence
    \[
        c_1(\cN{\cF})^2 = \frac{1}{m^2} f^*c_1(\cA)^2 \neq 0,
    \]
    a contradiction. Thus 
    \[
    \dim \Y \le 1.
    \]
    
    Let $\omega\in \coH[0]{\X,\forms[1]{\X}\otimes \cN{\cF}}$ be a twisted $1$-form defining $\cF$. Let $\F$ be a general fibre of $f$. By adjunction, $K_\F = K_{\X}|_\F$, and hence $-K_\F = (-K_{\X})|_\F$ is nef. Since $-K_{\X}$ is big, its restriction to a general fibre is big as well, so $-K_\F$ is nef and big. In other words, $\F$ is a weak Fano variety. Notice that the line bundle $\cN{\cF}|_\F$ is torsion since $\cN{\cF}^{\otimes m}\simeq f^*\cA$. This implies $\cN{\cF}|_\F\simeq \cO{\F}$ since weak Fano manifolds are simply connected by \cite{zhang_rcc}.
    Moreover $\omega|_{\F}=0 \in \coH[0]{\F,\forms[1]{\F}}$, and $\F$ is tangent to $\cF$.  
    As a consequence, $\dim \F \le \dim \X-1$, so $\dim \Y\ge 1$.
    Combining this with $\dim \Y\le 1$, we see that $\Y$ is a normal projective curve, hence smooth. This shows that $\cF$ is defined by the morphism $f:\X\to \Y$. 
    
    Because $f$ has connected fibres, the pull-back map $\coH[1]{\Y,\cO{\Y}}\to \coH[1]{\X,\cO{\X}}$ is injective, hence $\coh[1]{\Y,\cO{\Y}}=0$ and therefore $\Y\simeq \PP^1$.
    
    Finally, we note that since the fibres of $f$ are tangent to $\cF$, their singular loci are contained in $\sing(\cF)$.  Since $\codim \sing(\cF) \ge 3$, the fibres must therefore be irreducible.
\end{proof}

\section{Pfaff fields on compact K\"ahler manifolds}\label{sec:pfaff}

In this section, we discuss an analogue, for compact K\"ahler manifolds, of some of Esteves--Kleiman's work~\cite{EstevesKleiman}  on Pfaff fields on algebraic varieties.

\subsection{Pfaff fields}
Let $\X$ be a complex manifold and let $p$ be an integer such that $0 < p < \dim \X$.  Following Esteves--Kleiman~\cite{EstevesKleiman}, a \defn{Pfaff field of degree $p$ on $\X$} is a pair $(\cL,\eta)$ where $\cL$ is an invertible sheaf, and
\[
\eta : \forms[p]{\X} \to \cL
\]
is an $\cO{\X}$-linear map.  Equivalently, since $\X$ is assumed smooth, we may view $\eta$ as a section
\[
\eta \in \coH[0]{\X,\wedge^p\cT{\X}\otimes\cL} \cong \coH[0]{\X,\forms[n-p]{\X}\otimes \omega_\X^\vee\otimes \cL}
\]
where $\omega_\X$ denotes the canonical line bundle of $\X$ and $n=\dim \X$.  Note that $\eta$ defines a distribution
\[
\ker \eta \subset \cT{\X}
\]
as the set of vectors $v \in \cT{\X}$ such that $\eta \wedge v = 0$.  The \defn{singular locus of $\eta$} is the closed analytic subspace whose defining ideal is the image of the natural map $\forms[p]{\X}\otimes\cL^\vee \to \cO{\X}$ induced by $\eta$.

An analytic subspace $\Y\subset \X$ is \defn{$\eta$-invariant} if the natural map $\eta|_\Y : \forms[p]{\X}|_\Y \to \cL|_\Y$ factors through the quotient $\forms[p]{\X}|_\Y \to \forms[p]{\Y}$, i.e.~we have a (necessarily unique) dashed arrow $\eta_\Y$ filling in the following commutative diagram:
\[
\begin{tikzcd}
    \forms[p]{\X}\ar[r,"\eta"]\ar[d] & \cL \ar[d] \\
    \forms[p]{\Y} \ar[r,dashed,"\eta_\Y"] & \cL|_\Y.
\end{tikzcd}
\]

\subsection{The cohomological constraint of Esteves--Kleiman}

The key notion here is that of a closed strongly directed positive current. This is a slight generalization of \cite[Section~3]{numericallyflat}.

\begin{definition}
Let $\X$ be a complex manifold of dimension $n$, and let $\alpha\in \coH[0]{\X,\forms[q]{\X}\otimes \cL}$ be a twisted $q$-form. We say that a closed positive current $T$ of bidegree $(q,q)$ on $\X$ is \emph{strongly directed} (by $\alpha$) if one can locally write
		\[
		T={ ( \sqrt{-1})}^{q^2} a\ \omega\wedge \overline{\omega},
		\]
where $a$ is a positive locally finite Borel measure and $\omega$ is a local generator of $\alpha(\cL^\vee) \subseteq \forms[q]{\X}$ on some open analytic subset $\U$ of $\X$.
\end{definition}

Let $\Z$ be the zero locus of $\alpha$. Note that $\bar{\partial}(a \bar{\omega}) = 0$ on $\U \setminus \Z$ since $T$ is closed by assumption. We will need the following observation.

\begin{lemma}\label{lemma:closed}
We have $\bar{\partial}(a \bar{\omega}) = 0$ on $\U$.
\end{lemma}

\begin{proof}
The statement is local on $\U$, so we may shrink $\U$ and assume that $\U \cong \mathbb{D}^n$ is a polydisk, with coordinates $(z_1,\ldots,z_n)$. Set $p:= n -q$ and let $\eta$ be a $(p,q)$-form on $\U$, which we write as
\[
\eta = \sum_{I,J} f_{IJ} \, dz_I \wedge d\bar{z}_J
\]
for some $C^\infty$ functions $f_{IJ}$.  Let $\|\eta\| = \left( \sum_{|I|=p,|J|=q} |f_{IJ}|^2 \right)^{1/2}$ be the pointwise norm of $\eta$.
 
We likewise write 
\[
\omega = \sum_{|I|=q} \omega_I \, dz_I,
\]
so that
\[ 
\|\omega\| = \left( \sum_I |\omega_I|^2 \right)^{1/2} = \|\bar{\omega}\|. 
\]

We may also assume without loss of generality that $a(\Z \cap \U) = 0$ since $a \bar{\omega} = \chi_{\U \setminus \Z} a \bar{\omega}$. 

Let $\varphi_\epsilon \in C^\infty_c(\RR)$ for $\varepsilon > 0$ be a collection of bump functions that are equal to one in a neighbourhood of zero and such that $\textup{supp}(\varphi_\varepsilon) \subseteq [-\varepsilon,\varepsilon]$.  By \cite[Lemma~2.2]{HarveyPolking} (attributed there to Bochner), we can assume further that there is a constant $C>0$ such that
\begin{align}
\| \varphi'_\varepsilon \|_\infty \le \frac{C}{\varepsilon} \qquad \textrm{and} \qquad \| \varphi_\varepsilon \|_\infty \le C \label{eq:bochner}
\end{align}
for all $\varepsilon > 0$.

Let $\Phi_\varepsilon \colon \U \to \mathbb{R}$ be defined by $\Phi_\varepsilon = \varphi_\varepsilon(\|\omega\|)$, so that $\Phi_\varepsilon \in C^\infty(\U)$.
Let $\beta$ be a compactly supported $(n,p-1)$-form on $\U$.  Then
\[
\begin{array}{ccll}
\langle \bar{\partial}(a \bar{\omega}), \beta \rangle & = & \langle \bar{\partial}(a \bar{\omega}), \Phi_\varepsilon \beta \rangle & \textup{since } \textup{supp}(\bar{\partial}(a \bar{\omega})) \subseteq \Z \cap \U\\
& = & (-1)^{q+1} \langle a \bar{\omega}, (\bar{\partial} \Phi_\varepsilon) \wedge \beta + \Phi_\varepsilon \bar{\partial} \beta \rangle. & \\
\end{array}
\]

Notice that 
\[
a(\text{supp}(\Phi_\varepsilon)) \underset{\varepsilon \to 0}{\longrightarrow} a(\Z \cap \U) = 0
\] 
by construction, and that 
\[\| \bar{\omega} \wedge \bar{\partial} \Phi_\varepsilon \wedge \beta \| = O(1)\] 
by \eqref{eq:bochner}. This immediately implies 
\[
\langle \bar{\partial}(a \bar{\omega}), \beta \rangle = 0,\]
proving the lemma.
\end{proof}

The following result extends \cite[Lemma~3.11]{numericallyflat} to our current setting.

\begin{lemma}\label{lemma:pfaff-kahler}
Let $\X$ be a compact K\"ahler manifold of dimension $n$, and let $\alpha\in \coH[0]{\X,\forms[q]{\X}\otimes \cL}$ be a twisted $q$-form. Let $\eta \colon \forms[p]{\X} \longrightarrow \omega_\X\otimes \cL$ be the Pfaff field induced by $\alpha$, where $p:=n-q$. Let $\X^\circ$ be an open subset whose complement in $\X$ is a closed analytic subset $\Si$ of codimension at least $q+2$. Suppose that there exists a nonzero closed positive current $T^\circ$ on $\X^\circ$ strongly directed by $\alpha|_{\X^\circ}$. Then the map
    \[
        \coH[p]{\eta} : \coH[p]{\X,\forms[p]{\X}} \longrightarrow  \coH[p]{\X, \omega_\X\otimes \cL } 
    \]
is nonzero.
\end{lemma}

\begin{proof}
By \autoref{lemma:closed} above, the closed current $T^\circ$ can be viewed as an $\cL^\vee|_{\X^\circ}$-valued $\bar\partial$-closed current of degree $(0,q)$ on $\X^\circ$, giving a class
    \[
        c^\circ =[T^\circ]\in \coH[q]{\X^\circ,\cL^\vee|_{\X^\circ}}.
    \]
Moreover, the restriction morphism
    \[
        \coH[q]{\X,\cL^\vee}\longrightarrow
        \coH[q]{\X^\circ,\cL^\vee|_{\X^\circ}}
    \]
is an isomorphism since $\codim_\X \Si\ge q+2$ by assumption (see \cite[Chapter~2]{Banica74}). Let $c\in \coH[q]{\X,\cL^\vee}$ be the class mapping to $c^\circ$.
    
Since $\codim_\X \Si\ge q+1$, Harvey's removable singularities theorem \cite[Theorem~6]{Harvey74} applies to show that $T^\circ$ extends to a closed positive current $T$ on $\X$. Let $[T]\in \coH[q]{\X,\forms[q]{\X}}$ denote the class of the closed current $T$. By construction, $[T]$ is the image of $c$ under the map 
    \[
    \coH[q]{\X,\cL^\vee} \longrightarrow \coH[q]{\X,\forms[q]{\X}},
    \]
where $\cL^\vee\to \forms[q]{\X}$ is the morphism induced by $\alpha$.

    Finally, choose a K\"ahler form $\kahcl$ on $\X$. The image of $\coH[p]{\eta}(\kahcl^p)\otimes c$ under the natural map
    \[
    \coH[p]{\X,\omega_\X\otimes\cL} \otimes \coH[q]{\X,\cL^\vee} \longrightarrow \coH[n]{\X,\omega_\X}
    \]
    is $[\kahcl^p]\wedge [T]$. On the other hand, since $T$ is a nonzero positive current of bidegree $(q,q)$, we have
    \[
        \langle T, \kahcl^{p}\rangle \;>\; 0,
    \]
    so that $[\kahcl^p]\wedge [T]\neq 0$. This immediately implies that $\coH[p]{\eta}(\kahcl^p)\neq 0$, proving the lemma.
\end{proof}

The following result is an analogue, for compact K\"ahler manifolds, of \cite[Proposition 3.3]{EstevesKleiman}. 

\begin{proposition}\label{prop:esteves_kleiman}
Let $\X$ be a compact K\"ahler manifold of dimension $n$, and let
    \[
        \eta \colon \forms[p]{\X} \longrightarrow \cL
    \]
be a Pfaff field. Let $\Si\subset \Y \subset \X$ be closed subvarieties such that:
\begin{enumerate}
        \item $\Y$ is an $\eta$-invariant irreducible subvariety of dimension $p$, 
        \item $\Si$ has dimension at most $p-2$,
        \item $\Y^\circ:=\Y\setminus \Si$ is smooth,
        \item the restriction of $\eta$ to $\Y^\circ$ is nowhere vanishing. 
    \end{enumerate}
Then the map
    \[
        \coH[p]{\eta} : \coH[p]{\X,\forms[p]{\X}} \longrightarrow  \coH[p]{\X, \cL } 
    \]
    is nonzero.
\end{proposition}
    
\begin{proof}
Set $\X^\circ:=\X\setminus \Si$ and $q:=n-p$, and let $\alpha\in \coH[0]{\X,\forms[q]{\X}\otimes \omega_X^\vee\otimes \cL}$ be the twisted $q$-form induced by $\eta$. The current of integration along $\Y^\circ$ is a closed positive current on $\X^\circ$, which is strongly directed by $\alpha|_{\X^\circ}$ since the restriction of $\eta$ to $\Y^\circ$ is nowhere vanishing by assumption. The statement then follows from \autoref{lemma:pfaff-kahler}.
\end{proof}

\section{Degeneracy loci of Poisson structures}

\subsection{Poisson structures}\label{SS:Poissonstructures}
We recall a few basic notions from holomorphic Poisson geometry, mostly following \cite[\S2--\S7]{Gualtieri2013}.

Let $\X$ be a complex manifold and let $\sigma\in \coH[0]{\X,\wedge^2\cT{\X}}$ be a holomorphic Poisson structure, i.e.~$[\sigma,\sigma]=0$. We denote by
\[
    \sigma^\sharp\colon\forms[1]{\X}\longrightarrow \cT{\X}
\]
the morphism defined by contraction with $\sigma$. The image $\sigma^\sharp(\forms[1]{\X})\subset\cT{\X}$ is involutive.  Its saturation is a subsheaf
\[
\cF_\sigma \subset \cT{\X},
\]
defining a foliation that agrees with the symplectic foliation of $\sigma$ on the open dense set where the latter has maximal rank.  Note, however, that the leaves of $\cF_{\sigma}$ are, in general, unions of symplectic leaves of $\sigma$.

A \defn{Poisson subscheme} is a subscheme $\Y\subset \X$ to which $\sigma$ is tangent.  Equivalently, it is preserved by all vector fields of the form $\hook{\dd f}\sigma$ for $f \in \cO{\X}$ (the Hamiltonian vector fields).  It is called a \defn{strong Poisson subscheme}~\cite[Def.~4]{Gualtieri2013} if it is preserved by all germs of vector fields $\xi$ such that $\lie{\xi}\sigma = 0$ (the Poisson vector fields, which include the Hamiltonian vector fields).

For $k\ge 0$, the \defn{$2k$th degeneracy locus of $\sigma$} is the set
\[
    \Dgn[2k]{\sigma} :=\bigl\{x\in \X\mid \rank(\sigma_x^\sharp)\le 2k\bigr\},
\]
equipped with its natural scheme structure as the zero locus of 
\[
    \sigma^{k+1}\in \coH[0]{\X,\wedge^{2k+2}\cT{\X}}.
\]
It is a strong Poisson subscheme of $\X$ by \cite[Prop.~6]{Gualtieri2013}.

If $\Y \subset \X$ is a Poisson subscheme on which $\sigma$ has maximal rank $2k$, the \defn{Poisson singular locus} of $(\Y,\sigma)$ is the union
\[
\sing(\Y,\sigma) := \sing(\Y) \cup (\Dgn[2k-2]{\sigma} \cap \Y)
\]
of the singular locus of $\Y$, and the locus where the rank of $\sigma|_\Y$ drops below its maximal value.  If $\Y$ is a strong Poisson subscheme, so is $\sing(\Y,\sigma)$.

\subsection{Hodge numbers and leaf closures}

Let $\X$ be a compact K\"ahler manifold and let $\sigma$ be a holomorphic Poisson structure.  By an \defn{algebraic symplectic leaf} we mean a symplectic leaf of $\sigma$ whose closure is an irreducible closed subvariety of the same dimension (i.e.~a symplectic leaf whose Zariski closure has the same dimension).  Note that the closure of an algebraic symplectic leaf of dimension $2k$ is an invariant subvariety for the Pfaff field $\sigma^k$, to which we may apply \autoref{prop:esteves_kleiman}.  More generally, we may consider Poisson subvarieties whose singularities are sufficiently mild and whose foliations support strongly directed positive currents.   We deduce that the existence of such subvarieties with sufficiently mild singularities puts a constraint on the Hodge numbers:
\begin{proposition}\label{prop:leaves-hodge}
    Let $(\X,\sigma)$ be a compact K\"ahler Poisson manifold.  Suppose that $\X$ contains a closed Poisson subvariety $\Y\subset\X$ with the following properties:
    \begin{enumerate}
        \item $\sigma|_\Y$ generically has rank $2k$,
        \item $\dim(\sing(\Y,\sigma|_\Y)) \leq 2k-2$, and
        \item the induced symplectic foliation on $\Y \setminus \sing(\Y,\sigma|_\Y)$ supports a closed strongly directed positive current (which is automatic if $\dim \Y = 2k$, i.e.~$\Y$ is the closure of an algebraic symplectic leaf).
    \end{enumerate}
    Then $\coh[0,2k]{\X} >0$.
\end{proposition}
\begin{proof}
    Set $\Si = \sing(\Y, \sigma|_\Y)$, $\X^\circ := \X \setminus \Si$, and $\Y^\circ := \Y \setminus \Si$. Write $\cF^\circ$ for the symplectic foliation of $\sigma|_{\Y^\circ}$ on $\Y^\circ$, and let
    \[
        p := 2k, \qquad r := \dim \Y - p,
    \]
    so that $r$ is the codimension of $\cF^\circ$ in $\Y^\circ$.

    By hypothesis~(3), there exists a closed positive current $T^\circ$ on $\Y^\circ$ strongly directed by a twisted $r$-form
    $\alpha^\circ$ defining $\cF^\circ$; locally on $\Y^\circ$,
    \begin{equation}\label{eq:tau-kahler}
        T^\circ = a \, (\sqrt{-1})^{r^2} \, \alpha^\circ \wedge \overline{\alpha^\circ},
    \end{equation}
    where $a$ is a positive locally finite Borel measure. In particular, $T^\circ$ has bidimension $(p, p)$ on $\Y^\circ$.

    Let $\Theta^\circ$ denote the pushforward of $T^\circ$ along the inclusion $\Y^\circ \hookrightarrow \X^\circ$:
    \[
       \langle \Theta^\circ, \varphi \rangle := \langle T^\circ, \varphi|_{\Y^\circ} \rangle
    \]
    for any compactly supported smooth test form $\varphi$ of bidegree $(p, p)$ on $\X^\circ$. Then $\Theta^\circ$ is a closed positive current on $\X^\circ$, supported on $\Y^\circ$, of bidimension $(p, p)$ (equivalently, of bidegree $(n - p, n - p)$). The local expression~\eqref{eq:tau-kahler} pushes forward to the local expression of $\Theta^\circ$ as a positive measure multiple of $\omega \wedge \overline{\omega}$, where $\omega$ is the local generator of $\cL^\vee \subseteq \forms[n - p]{\X^\circ}$ induced by $\sigma^k$; hence $\Theta^\circ$ is strongly directed by $\sigma^k$.

    Since $\codim_\X \Si \geq (n - p) + 2$ by hypothesis~(2), \autoref{lemma:pfaff-kahler} applies and gives $h^{0, 2k}(\X) > 0$.
\end{proof}

\subsection{Modular residues}
\label{sec:modres}
In \cite[\S4.3]{Gualtieri2013}, Gualtieri and the third author introduced the \defn{modular residues}
\[
    \Resmod[k]{\sigma}\in \coH[0]{\Dgn[2k]{\sigma},\der[2k+1]{\Dgn[2k]{\sigma}}},
\]
for $k \ge 0$, where $\der[\bullet]{\Dgn[2k]{\sigma}} := (\forms{\Dgn[2k]{\sigma}})^\vee$ is the sheaf of multiderivations on $\Dgn[2k]{\sigma}$.  These residues can be defined locally as follows.  For a local trivialization $\mu \in \omega_{\X}$ over some open set $\U\subset\X$, recall that the \defn{modular vector field} of $\sigma$ with respect to $\mu$ is the vector field $Z$ determined uniquely by the formula
\[
\hook{Z}\mu = - \dd(\hook{\sigma}\mu).
\]
Then the $k$th modular residue is given by the formula
\[
i_*\Resmod[k]{\sigma|_\U} = (Z \wedge \sigma^k)|_{\Dgn[2k]{\sigma} \cap \U},
\]  
where
\[
i_* : \der{\Dgn[2k]{\sigma}} \hookrightarrow \wedge^{\bullet}\cT{\X}|_{\Dgn[2k]{\sigma}}
\]
is the pushforward of multiderivations along the inclusion $i : \Dgn[2k]{\sigma}\to\X$.

One of the primary uses of the modular residues is the following lemma; it is a special case of \cite[Lemma~16]{Gualtieri2013}, in which the ample Poisson module is the anticanonical bundle.
\begin{lemma}[{\cite[Lemma~16]{Gualtieri2013}}]\label{lemma:GP64}
    Let $(\X,\sigma)$ be a Poisson Fano manifold, let $k$ be a positive integer and let $\Y$ be an irreducible, closed, strong Poisson subscheme of $\Dgn[2k]{\sigma}$ with the following properties: 
    \begin{enumerate}
        \item $\Y$ is not contained in $\Dgn[2k-2]{\sigma}$, and
        \item $\Y$ is contained in the vanishing locus of
        \[
        i_*\Resmod[k]{\sigma} \in \coH[0]{\Dgn[2k]{\sigma},\wedge^{2k+1}{\cT{\X}}|_{\Dgn[2k]{\sigma}}}.
        \]
    \end{enumerate}
    Then the Poisson singular locus $\sing(\Y,\sigma|_{\Y})$ is non-empty and has a component of dimension at least $2k-1$.
\end{lemma}

\subsection{Submaximal degeneracy loci on odd-dimensional Fanos}
We now combine the results of the previous sections to establish our main theorem (stated as \autoref{theorem:main} in the introduction): 
\begin{theorem}\label{thm:bondal-odd}
    Let $\X$ be a Fano manifold of dimension $2k+1$ for some $k > 1$, and let $\sigma$ be a Poisson structure on $\X$ of maximal rank $2k$.  Then $\Dgn[2k-4]{\sigma}$ has an irreducible component of dimension at least $2k-3$.
\end{theorem}
\begin{proof}
    Let $\cF = \cF_\sigma$ be the saturated foliation defined by $\sigma$.  Note that we have $\codim(\sing(\cF)) \ge 2$.

    If $\codim (\sing(\cF)) \ge 3$, then by \autoref{theorem:fibration}, $\cF$ is the relative tangent sheaf of a morphism $f : \X \to \PP^1$.  A general fibre $\F$ is then a Poisson subvariety of codimension one.  By the adjunction formula, $\F$ is a Fano manifold of even dimension, and hence the result follows from \cite[Theorem 27]{Gualtieri2013}.

    It remains to treat the case in which $\sing(\cF)$ has an irreducible component $\Y$ of codimension two in $\X$ (i.e.~$\dim \Y = 2k-1$), which we equip with the reduced scheme structure.  Let $\Si := \sing(\Y,\sigma|_\Y)$ be its Poisson singular locus.  Note that $\Si$ and $\Y$ are strong Poisson subvarieties of $\Dgn[2k-2]{\sigma}$, and $\dim \Si \leq 2k-2$.  We treat several cases according to the dimension of $\Si$.
        
    Suppose first that $\dim \Si = 2k-3$. Then $\sigma|_\Si$ must have rank at most $2k-4$.  Hence $\Si \subset \Dgn[2k-4]{\sigma}$ and the result follows. 

    Next, suppose that $\dim \Si = 2k-2$, and let $\Z$ be an irreducible component of $\Si$ of dimension $2k-2$. If the generic rank of $\sigma$ along $\Z$ is less than $2k-2$, then $\Z \subset  \Dgn[2k-4]{\sigma}$ and the result follows.  Otherwise, $\Z$ is the closure of an algebraic symplectic leaf of dimension $2k-2$.  Since $\X$ is Fano, we have $\coh[0,2k-2]{\X} =0$.  Hence by \autoref{prop:leaves-hodge}, the Poisson singular locus of $\Z$ has dimension at least $2k-3$, so it must be contained in $\Dgn[2k-4]{\sigma}$, and again, the result follows.

    We claim that the cases examined above are the only ones possible, i.e.~that the situation $\dim \Si \leq 2k-4$ cannot occur.
    
 Indeed, assume by contradiction that $\dim \Si \leq 2k - 4$.  Then $\Y^\circ := \Y \setminus \Si$ is smooth of dimension $2k - 1$ and $\sigma|_{\Y^\circ}$ is regular of rank $2k - 2$, i.e.\ the symplectic foliation of $\Y^\circ$ is regular of codimension one. Let $\rho := \Resmod[k-1]{\sigma}|_{\Y^\circ}$ be the restriction of the modular residue to $\Y^\circ$; it is a section of the anticanonical line bundle of $\Y^\circ$.

We first show that the vanishing locus of $\rho$ is nonempty. Suppose, on the contrary, that $\rho$ never vanishes. Then $\rho^{-1}$ is a nonvanishing holomorphic volume form on $\Y^\circ$, and the contraction $\alpha := \iota_{\sigma^{k-1}} \rho^{-1}$ is a nonvanishing holomorphic one-form defining the symplectic foliation of $\Y^\circ$. We claim that $\alpha$ is closed. To see this, let $Z$ be the modular vector field of $\sigma|_{\Y^\circ}$ with respect to $\rho^{-1}$. Since $\rho = \Resmod[k-1]{\sigma}|_{\Y^\circ}$ is nowhere vanishing, the projection of $Z$ to the normal bundle of the symplectic foliation is nowhere zero; hence Hamiltonian vector fields together with $Z$ span $\T_{\Y^\circ}$. Moreover, $\alpha$ vanishes on Hamiltonian vector fields and $\alpha(Z)$ is a nonzero constant. To verify that $\dd\alpha = 0$ it is enough to consider the formula
\[
    \dd\alpha(X, Y) = X(\alpha(Y)) - Y(\alpha(X)) - \alpha([X, Y]),
\]
when $X, Y$ are both Hamiltonian, or when $X$ is Hamiltonian and $Y = Z$. Either way, the first two terms vanish because the functions $\alpha(X)$ and $\alpha(Y)$ are constant, and the third term vanishes because $[X, Y]$ is Hamiltonian (being the bracket of a Hamiltonian and a Poisson vector field). Thus $\dd\alpha = 0$, as claimed. The product $\sqrt{-1}\,\alpha \wedge \overline{\alpha}$ therefore defines a closed strongly directed positive current of bidegree $(1, 1)$ on $\Y^\circ$, and \autoref{prop:leaves-hodge} gives $h^{0, 2k-2}(\X) > 0$, contradicting the assumption that $\X$ is Fano.

Hence the vanishing locus of $\rho$ is nonempty. Set  $\bar\rho=\Resmod[k-1]{\sigma}|_{\Y}$. By \cite[Prop.~7.5 and Lemma~3.4]{Gualtieri2013}, the scheme-theoretic vanishing locus of $\bar\rho$ and each of its reduced irreducible components are strong Poisson subschemes of $\Dgn[2k-2]{\sigma}$. By assumption, one can choose such a component $\W$ so that $\W\cap \Y^\circ\not=\emptyset$. Note that $\W=\Y$  when $\rho$ vanishes identically and $\W$ has dimension $2k - 2$ otherwise. 
In both cases, the singular Poisson locus  $\sing(\W,\sigma|_\W)$, which is a strong Poisson subscheme, is contained in $\Si$.
Applying \autoref{lemma:GP64} to $\W$     
yields a component of $\Si$ of dimension at least $2k - 3$, thus obtaining the sought contradiction.
\end{proof}

\bibliographystyle{hyperamsplain}
\bibliography{bondal}

\end{document}

%% file: bondal-preamble.tex
\usepackage{titling}

\usepackage{amsfonts,amsmath,amssymb,amsthm}
\usepackage{mathscinet}
\usepackage{mathrsfs}

\usepackage{combelow}

\usepackage{enumitem}
\setitemize{itemsep=0em}

\usepackage{multirow}

\usepackage[font=small,labelfont=bf]{caption}

\usepackage{todonotes}

\newcommand{\email}[1]{\href{mailto:#1}{#1}}

\usepackage[pdfusetitle]{hyperref}
\usepackage{color}
\definecolor{darkred}{rgb}{.8,0,0}
\definecolor{tocolor}{rgb}{.1,.1,.1}
\definecolor{urlcolor}{rgb}{.2,.2,.6}
\definecolor{linkcolor}{rgb}{.1,.1,.5}
\definecolor{citecolor}{rgb}{.4,.2,.1}
\definecolor{gray}{rgb}{.8,.8,.8}

\hypersetup{
	colorlinks=true,
	urlcolor=urlcolor,
	linkcolor=linkcolor,
	citecolor=citecolor
	}

\usepackage{aliascnt}

\newcommand{\thdef}[2]{
	\newaliascnt{#1}{theorem}
	\newtheorem{#1}[#1]{#2}
	\aliascntresetthe{#1}
	\newtheorem*{#1*}{#2}
	\expandafter\newcommand\expandafter{\csname #1autorefname\endcsname}{#2}
}

\newtheorem{theorem}{Theorem}[section]
\newtheorem*{theorem*}{Theorem}
\thdef{conjecture}{Conjecture}
\thdef{lemma}{Lemma}
\thdef{corollary}{Corollary}
\thdef{proposition}{Proposition}
\theoremstyle{definition}
\thdef{definition}{Definition}

\theoremstyle{remark}
\thdef{remarkx}{Remark}
\thdef{examplex}{Example}
\thdef{setup}{Setup}
\thdef{claim}{Claim}

\newcommand{\defn}[1]{\textbf{\emph{#1}}}

\usepackage{tikz,tikz-cd}

\newcommand{\PP}{\mathbb{P}}
\newcommand{\RR}{\mathbb{R}}

\newcommand{\C}{\mathsf{C}}
\newcommand{\T}{\mathsf{T}}
\newcommand{\F}{\mathsf{F}}

\newcommand{\W}{\mathsf{W}}

\newcommand{\Z}{\mathsf{Z}}
\newcommand{\Y}{\mathsf{Y}}

\newcommand{\X}{\mathsf{X}}
\newcommand{\Si}{\mathsf{S}}

\newcommand{\U}{\mathsf{U}}

\newcommand{\Dgn}[2][]{\mathsf{D}_{#1}(#2)}

\newcommand{\kahcl}{\gamma}

\newcommand{\sing}{\mathrm{sing}}

\DeclareMathOperator{\codim}{codim}

\DeclareMathOperator{\rank}{rank}

\DeclareMathOperator{\NE}{NE}

\newcommand{\cT}[1]{\mathcal{T}_{#1}} 
\newcommand{\cN}[1]{\mathcal{N}_{#1}} 
\newcommand{\coN}[1]{\mathcal{N}^\vee_{#1}} 

\newcommand{\der}[2][\bullet]{\mathscr{X}^{#1}_{#2}} 
\newcommand{\forms}[2][\bullet]{\Omega^{#1}_{#2}} 
\newcommand{\cA}{\mathcal{A}}
\newcommand{\cF}{\mathcal{F}}

\newcommand{\cL}{\mathcal{L}}

\newcommand{\cO}[1]{\mathcal{O}_{#1}} 

\newcommand{\Resmod}[2][]{\mathrm{Res}^{\mathrm{mod}}_{#1}(#2)}

\newcommand{\coH}[2][\bullet]{\mathsf{H}^{#1}\left(#2\right)}
\newcommand{\coh}[2][\bullet]{\mathsf{h}^{#1}(#2)}

\newcommand{\dd}{\mathrm{d}}

\newcommand{\lie}[1]{\mathscr{L}_{#1}}
\newcommand{\hook}[1]{\iota_{#1}}

\newcommand\restr[2]{{
  \left.\kern-\nulldelimiterspace 
  #1 
  \vphantom{\big|} 
  \right|_{#2} 
  }}